\documentclass[a4paper,12pt]{amsart}
\usepackage{tabularx}
\usepackage[dvips]{graphicx}
\usepackage{amsmath}
\usepackage{amssymb}
\usepackage{amsbsy}

\usepackage{amsgen}
\usepackage{amsopn}
\usepackage{amscd}
\usepackage{amstext}
\usepackage{amsthm}

\theoremstyle{definition}
\newtheorem{definition}{Definition}[section]
\theoremstyle{plain}
\newtheorem{theorem}{Theorem}[section]

\newtheorem{lemma}{Lemma}[section]

\newtheorem{proposition}{Proposition}[section]

\theoremstyle{remark}
\newtheorem{remark}{Remark}[section]

\theoremstyle{definition}

\newtheorem{acknow}{Acknowledgment}

\begin{document}
\title{HOMOTOPY CLASSIFICATION OF GENERALIZED PHRASES 
IN TURAEV'S THEORY OF WORDS}

\author{FUKUNAGA Tomonori}\thanks{The author is JSPS research fellow (DC)}

\maketitle

\begin{abstract}
In 2005 V. Turaev introduced the theory of topology of words and phrases.
Turaev defined an equivalence relation on 
generalized words and phrases which is called homotopy.
This is suggested by the Reidemeister moves in the knot theory.
Then Turaev gave the homotopy classification of generalized words
with less than or equal to five letters.
In this paper we give the classification of generalized phrases
up to homotopy with less than or equal to three letters.
To do this we construct a new homotopy invariant
for nanophrases over any $\alpha$.      
\end{abstract}

{\bf keywords:} words, phrases, homotopy of words and phrases 

\markboth{FUKUNAGA Tomonori}{Homotopy Classification of Generalized Phrases}

\section{Introduction.}
In \cite{tu1} and \cite{tu2}, 
V. Turaev introduced the theory of topology of 
words and phrases. Words are finite sequences of letters
in a given alphabet, letters are elements of an alphabet and phrases
are finite sequences of words. Turaev defined generalized words 
which is called \'etale words as follows : Let $\alpha$ be an alphabet
endowed with an involution $\tau:\alpha \rightarrow \alpha$. Let
$\mathcal{A}$ be an alphabet endowed with a mapping 
$|\cdot|:\mathcal{A} \rightarrow \alpha$ which is called a projection.
We call this $\mathcal{A}$ an $\alpha$-alphabet. Then we call a
pair an $\alpha$-alphabet $\mathcal{A}$ and a word on $\mathcal{A}$ 
an \'etale word.
If all letters in $\mathcal{A}$ appear exactly twice, then we call
this \'etale word a nanoword.\par
Turaev introduced an equivalence relation which is called homotopy on 
nanowords. This equivalence relation is suggested 
by the Reidemeister moves in the theory of knots.
Homotopy of nanowords is generated by isomorphism, 
and three homotopy moves. 
The first homotopy move is deformation
that changes $xAAy$ into $xy$. The second homotopy move is deformation that
changes $xAByBAz$ into $xyz$ when $|A|$ is equal to $\tau(|B|)$. 
The third homotopy move is deformation 
that changes $xAByACzBCt$ into $xBAyCAzCBt$ when $|A|$ and $|B|$ are equal
to $|C|$ (Turaev defined more generalized equivalence relation
which is called $S$-homotopy. However, in this paper, we treat only 
homotopy). Turaev defined homotopy of \'etale words via desingularization 
of \'etale words. 
Moreover in \cite{tu2} Turaev defined homotopy of nanophrase
in a similar manner.\par
Theory of words and phrases is applied for studying curves on surfaces.
In \cite{ga} C. F. Gauss studied planar curves via words.
Turaev applied generalized words and phrases for curves and knot
diagrams.
Turaev showed  special cases of the theory of topology of phrases 
corresponds to 
the theory of 
stable equivalent classes of ordered, pointed, oriented multi-component
curves on surfaces and knot diagrams. 
Note that the theory 
stable equivalence classes of ordered, pointed, oriented multi-component
curves on surfaces (respectively knot diagrams) is equivalent to the theory
of ordered open flat virtual links (respectively ordered open virtual
links). In this paper, 
ordered links means each components of links are numerated.
See also \cite{kad}, \cite{ka}, \cite{sw} and \cite{tu5} for more
details.
In this meaning, the theory of topology of words and phrases is 
combinational extension of the theory of virtual knots and links.\par
Now the purpose of this paper is classification of generalized
phrases (in this paper we call it \'etale phrases) up to homotopy.
Turaev gave the homotopy classification of \'etale words with 
less than or equal to five letters in \cite{tu1}.
We will extend this result. More precisely, 
in this paper, we give the classification of \'etale phrases
with less than or equal to three letters.
To do this we use some known invariants which was introduced in
\cite{fu1}, \cite{fu2} and \cite{gi}. 
Moreover we construct a new homotopy invariant.\par
The rest of this paper is constructed as follows.
In the next section 
we review the theory of topology of nanowords, \'etale words,
nanophrases and \'etale phrases.  
In Section 3 we introduce homotopy invariants of nanophrases which
was introduced in \cite{fu1}, \cite{fu2} and \cite{gi}. Then we will
define a new homotopy invariant for nanophrases. 
In Section 4 and Section 5 we give the classification of 
\'etale phrases with less than or equal to three letters 
without the condition on length of phrases.  
               
\section{\'Etale Phrases  and Nanophrases.}
In this section we introduce Turaev's theory of words and phrases
(See \cite{tu1}, \cite{tu2} and \cite{tu3} for more details).
\subsection{\'Etale words and \'etale phrases. }
 In this paper an \emph{alphabet} means a finite set and
\emph{letters} mean its element. A \emph{word} of length $n$ on an
alphabet $\mathcal{A}$ is a mapping $w:\hat{n} \rightarrow \mathcal{A}$
where $\hat{n} := \{1 , 2 , \cdots n \}$ and a \emph{phrase} of 
length $k$ on $\mathcal{A}$ is a finite sequence of words on $\mathcal{A}$,
$(w_{1}|w_{2}|\cdots|w_{k})$. 
A multiplicity of a letter $A \in \mathcal{A}$ in a phrase $P$ on 
$\mathcal{A}$ is a number of $A$ in the phrase $P$. We denote 
multiplicity of $A \in \mathcal{A}$ by $m_{P}(A)$.\par
Let $\alpha$ be an alphabet endowed with an involution 
$\tau:\alpha \rightarrow \alpha$. An \emph{$\alpha$-alphabet} is a
pair (An alphabet $\mathcal{A}$, 
mapping $|\cdot| : \mathcal{A} \rightarrow \alpha$).
We call the mapping $|\cdot|$ \emph{projection}.\par
In \cite{tu5}, V. Turaev defined generalized words
which is called \'etale words. An \emph{\'etale word} over $\alpha$ is a
pair (An $\alpha$-alphabet $\mathcal{A}$, A word on $\mathcal{A}$) and
A \emph{\'etale phrase} over $\alpha$ is a pair 
(An $\alpha$-alphabet $\mathcal{A}$, A phrase on $\mathcal{A}$).
\begin{remark}
Turaev did not define  \'etale phrases explicitly. 
However Turaev considered an equivalent object in \cite{tu2}.     
\end{remark}

A phrase $P$ on an $\alpha$ gives rise to an \'etale phrase $(\alpha, P)$
 where the projection $\alpha \rightarrow \alpha$ is the identity mapping.
In this meaning \'etale phrases are generalization of usual phrases.
\subsection{Nanowords and  nanophrases.}   
A \emph{Gauss word} on an alphabet $\mathcal{A}$ is a word $w$ on $\mathcal{A}$
which all letters in $\mathcal{A}$ appear exactly twice in $w$. 
A phrase $P$ on $\mathcal{A}$ is called a \emph{Gauss phrase}
if all letters in $\mathcal{A}$ appear exactly twice in $P$.\par
In this paper, we consider generalized Gauss words and Gauss phrases.
 A \emph{nanoword} over $\alpha$ is a
pair (An $\alpha$-alphabet $\mathcal{A}$, A Gauss word on $\mathcal{A}$) and
A \emph{nanophrase} over $\alpha$ is a pair 
(An $\alpha$-alphabet $\mathcal{A}$, A Gauss phrase on $\mathcal{A}$). 
Instead of writing $(\mathcal{A},P)$ for a nanophrase over $\alpha$, 
we often write simply $P$. The alphabet $\mathcal{A}$ 
can be uniquely recovered. However the projection 
$|\cdot|:\mathcal{A} \rightarrow \alpha$ should be always specified.
\subsection{Desingularization of  \'etale phrases.}
In this section, we introduce a method of associating with 
any \'etale phrases over $\alpha$ $(\mathcal{A}, P)$ a nanophrase
over $\alpha$ $(\mathcal{A}^{d}, P^{d})$ which is called 
\emph{desingularization} of \'etale phrases.\par
Let $\mathcal{A}^d$ be an $\alpha$-alphabet 
$\{A_{i,j}:=(A,i,j) |A \in \mathcal{A}, 1 \le i < j \le
 m_P(A) \}$ with the projection $|A_{i,j}|:=|A|$ for all
 $A_{i,j}$. 
The phrase $P^d$ is obtained from $P$ by first deleting all $A \in
 \mathcal{A}$ with $m_P(A)$ is less than or equal to one. 
Then for each $A \in \mathcal{A}$ with
 $m_P(A)$ is grater than or equal to two and each $i = 1,2, \ldots m_P(A)$, 
we replace the $i$-th
 entry of $A$ in $P$ by $$A_{1,i}A_{2,i} \ldots
 A_{i-1,i}A_{i,i+1}A_{i,i+2} \ldots A_{i,m_P(A)}.$$
The resulting $(\mathcal{A}^d,P^d)$ is a nanophrase with $\sum
m_P(A)(m_P(A)-1)$ letters and called a 
\emph{desingularization of $(\mathcal{A},P)$}.
Note that if $(\mathcal{A},P)$ is a nanophrase, then desingularization
of $(\mathcal{A},P)$ is a itself. 
\subsection{Homotopy of nanophrases and \'etale phrases.}
In \cite{tu1} and \cite{tu2}, 
Turaev defined an equivalence relation which is called 
homotopy on a set of nanophrases and \'etale words.\par
To define homotopy, we define isomorphism of \'etale phrases.
 A \emph{morphism} of $\alpha$-alphabets
$\mathcal{A}_1$, $\mathcal{A}_2$ is a set-theoric mapping 
$f:\mathcal{A}_1 \rightarrow \mathcal{A}_2$ such that $|A|=|f(A)|$ 
for all $A \in \mathcal{A}_1$. If $f$ is bijective, then this morphism 
is an \emph{isomorphism}. 
Two \'etale phrases $(\mathcal{A}_1,(w_{1}|\cdots|w_{k})$ and
$(\mathcal{A}_2,(v_{1}|\cdots|v_{k}))$ over $\alpha$ are 
\emph{isomorphic} if there is 
an isomorphism $f:\mathcal{A}_1 \rightarrow \mathcal{A}_2$ such that 
$v_{j} = f \circ w_{j}$ for all $j \in \hat{k}$.\par
Next we define homotopy moves of nanophrases.
\begin{definition}
 We define \textit{homotopy moves} (1) - (3) of nanophrases as follows:
 \par
 (1) $(\mathcal{A} , (xAAy)) \longrightarrow 
(\mathcal{A} \setminus \{ A \} , (xy))$ 
 \par \enskip \enskip \enskip 
for all $A \in \mathcal{A}$ and $x,y$ are sequences of letters in 
$\mathcal{A} \setminus \{ A \}$, possibly including \par
\enskip \enskip \enskip the $|$ character.
 \par
(2) $(\mathcal{A} , (xAByBAz)) \longrightarrow (\mathcal{A} \setminus \{ A , B \} , (xyz))$
   \par \enskip \enskip \enskip 
if $A , B \in \mathcal{A}$ satisfy $|B| = \tau (|A|)$. $x,y,z$ are sequences of
letters in 
$\mathcal{A} \setminus \{A,B\}$, \par
\enskip \enskip \enskip possibly including $|$ character.
 \par
(3) $(\mathcal{A} , (xAByACzBCt)) \longrightarrow (\mathcal{A} , (xBAyCAzCBt))$
 \par \enskip \enskip \enskip if $A,B,C \in \mathcal{A}$ 
satisfy $|A|=|B|=|C|$. $x,y,z,t$ are sequences of letters in \par
\enskip \enskip \enskip $\mathcal{A}$, possibly including $|$ character.
\end{definition}
Now we define homotopy of \'etale phrases.
\begin{definition}
 Two \'etale phrases $(\mathcal{A}_1 ,P_1)$ and $(\mathcal{A}_2 , P_2)$
 over $\alpha$ are \textit{homotopic}
 (denoted $(\mathcal{A}_1 , P_1) \simeq (\mathcal{A}_2 , P_2)$)
 if $((\mathcal{A}_2)^{d} , (P_2)^{d})$ can be obtained from 
$((\mathcal{A}_1)^{d}, (P_1)^{d})$ by a
 finite sequence of isomorphism, homotopy moves 
 (1) - (3) and the inverse of moves (1) - (3).
\end{definition}

\begin{remark}
By the definition of homotopy of \'etale phrases, 
every homotopy invariant $I$ of nanophrases extends to 
a homotopy invariant $I$ of \'etale phrases by 
$I(P):=I(P^{d})$.
\end{remark}

The gale of this paper is to classify \'etale phrases of length $k$
over $\alpha$ with less than or equal to three letters up to homotopy
for any $k$ and $\alpha$.\par
The case of $k$ is equal to one (in other words, the case of 
\'etale words) Turaev gave the classification as follows.
\begin{theorem}[Turaev \cite{tu1}]\label{tu}
A multiplicity-one-free word of length less than or equal to
four in the alphabet $\alpha$ has one of the following forms:
$aa$, $aaa$, $aaaa$, $aabb$, $abba$, $abab$ with distinct $a,b \in \alpha$
The words $aa$, $aabb$, $abba$ are contractible. The words $aaa$ and
$aaaa$ are contractible if and only if $\tau(a)=a$. The word $abab$
is contractible if and only if $\tau(a)=b$. Non-contractible words of 
type $aaa$, $aaaa$ and $abab$ are homotopic if and only if 
they are equal.   
\end{theorem}
\section{Homotopy Invariants of Nanophrases.}
By the definition of homotopy of \'etale phrases, 
we need homotopy invariants of nanophrases.
In this section, we introduce homotopy invariants of nanophrases
which were defined in \cite{fu1}, \cite{fu2} and \cite{gi}. 
Moreover we define a new homotopy invariant of nanophrases.
\subsection{Simple invariants.}
In this subsection, we review homotopy invariants which were defined in
\cite{fu2} and \cite{gi}.\par
Let $P = (w_{1}|w_{2}| \cdots |w_{k})$ be a nanophrase over $\alpha$. 
For $l \in \hat{k}$, we define $w(l) \in \mathbb{Z}/2\mathbb{Z}$ by 
the length of $w_{l}$. We call the vector 
$$w(P):=(w(1), \cdots ,w(k)) \in (\mathbb{Z}/2\mathbb{Z})^{k}$$ 
the \emph{component length vector}.
\begin{proposition}[A. Gibson \cite{gi}, see also \cite{fu1}]\label{clv}
The component length vector is a homotopy invariant of nanophrases.
\end{proposition}  
Next we define another homotopy invariant.
Let $\pi$ be the group which is defined as follows:
$$\pi := (a \in \alpha|a\tau(a)=1 , ab = ba \ for \ all 
\ a, b \in \alpha \ ).$$ 
Let $(w_1|w_2|\cdots|w_k)$ be a nanophrase of
length $k$ over $\alpha$.
We define $l_{P}(i,j) \in \pi$ for $i < j$
by
$$l_{P}(i,j) := \prod_{A \in Im(w_i) \cap Im(w_j)} |A|.$$ 
We call a vector 
$lk(P) := (l_{P}(1,2),l_{P}(1,3),\cdots,l_{P}(1,k),l_{P}(2,3),
\cdots,l_{P}(k-1,k)) \in \pi^{\frac{1}{2}k(k-1)}$ 
the \emph{linking vector}. 
\begin{proposition}[\cite{fu2}]\label{lv}
The linking vector of nanophrases is a homotopy invariant of nanophrases.
\end{proposition}

\subsection{The invariant $T$.}
In this section we introduce a homotopy invariant $T$ which was
defined by the author in \cite{fu1}. This invariant is defined for nanophrases
over $\alpha_{0}$ and the one-element set 
where $\alpha_{0}=\{a,b\}$ with the involution $\tau_{0}:a \mapsto b$.

\begin{definition}
Let $P=(\mathcal{A},(w_1|\cdots|w_k))$ be a nanophrase
over $\alpha_0$ and $A$,$B \in \mathcal{A}$. Then we define
$\sigma_P(A,B)$ as follows: If $A$ and $B$ form
$\cdots A \cdots B \cdots A \cdots B \cdots$ in $P$ and $|B|=a$, or 
$\cdots B \cdots A \cdots B \cdots A \cdots$ in $P$ and $|B|=b$, then
$\sigma_P(A,B):= 1$. If $\cdots A \cdots B \cdots A \cdots B \cdots$ in
 $P$ and $|B|=b$, or $\cdots B \cdots A \cdots B \cdots A \cdots$
 in $P$ and $|B|=a$, then $\sigma_P(A,B):=-1$. Otherwise $\sigma_P(A,B):=0$.  
\end{definition}

\begin{definition}
For $A \in \mathcal{A}$ we define $\varepsilon(A) \in \{ \pm 1 \} $ by
$$\varepsilon(A):=\begin{cases} 1 \ ( \ if \ |A|=a \  ),
 \\ -1 \ ( \ if \ |A|=b \ ). \end{cases}$$
\end{definition}

\begin{definition}
Let $P=(\mathcal{A}, (w_1|w_2|\cdots|w_k))$ be a nanophrase of
length $k$ over $\alpha_0$. For $A \in \mathcal{A}$ such that
 there exist $i \in \{1,2,\cdots,k\}$ such that $Card(w_i^{-1}(A))=2$, 
we define $T_P(A) \in \mathbb{Z}$ by
$$T_P(A):= \varepsilon(A)\sum_{B \in \mathcal{A}}\sigma_P(A,B),$$
and we define $T_P(w_i) \in \mathbb{Z}$ by
$$T_P(w_i):= \sum_{A \in \mathcal{A},\ Card(w_i^{-1}(A))=2 }T_P(A).$$
Then we define $T(P) \in \mathbb{Z}^k$ by 
$$T(P):=(T_P(w_1),T_P(w_2),\cdots,T_P(w_k)).$$
\end{definition}

\begin{theorem}[\cite{fu1}]
 $T$ is a homotopy invariant of nanophrases over $\alpha_0$.
\end{theorem}

Next we define an invariant $T$ for nanophrases over the one-element set
(we use the same notation "$T$" because of the Remark \ref{rem1}).\par

\begin{definition}
Let $P:=(\mathcal{A},(w_1|\cdots|w_k))$ be a nanophrase
 over the one-element set $\alpha:=\{a\}$.
Let $A,B \in \mathcal{A}$ be letters. 
Then we define $\tilde{\sigma}_{P}(A,B) \in \mathbb{Z}/2\mathbb{Z}$ 
as follows: If $A$ and $B$
 forms $\cdots A \cdots B \cdots A \cdots B \cdots$ or  
$\cdots B \cdots A \cdots B \cdots A \cdots$ in P, 
then  $\tilde{\sigma}_{P}(A,B):=1$. Otherwise  $\tilde{\sigma}_{P}(A,B):=0$.
\end{definition}

\begin{definition}
Let $P:=(\mathcal{A},(w_1|\cdots|w_k))$ be a nanophrase over
$\alpha:=\{a\}$. For $A \in \mathcal{A}$ such that there exist an
 $i \in \{1,2,\cdots,k\}$ such that $Card(w_i^{-1}(A))=2$, we define
$T_P(A) \in \mathbb{Z}/2\mathbb{Z}$ by 
$$T_P(A):=\sum_{B \in \mathcal{A}}\tilde{\sigma}_P(A,B) \in
 \mathbb{Z}/2\mathbb{Z},$$
and $T_P(w_i) \in \mathbb{Z}/2\mathbb{Z}$ by
$$T_P(w_i):= \sum_{A \in \mathcal{A},\ Card(w_i^{-1}(A))=2 }T_P(A).$$
Then we define $T(P) \in (\mathbb{Z}/2\mathbb{Z})^k$ by 
$$T(P):=(T_P(w_1),T_P(w_2),\cdots,T_P(w_k)).$$
\end{definition}

Then the next theorem follows.

\begin{theorem}[\cite{fu1}]
 $T$ is a homotopy invariant of nanophrases over the one-element set. 
\end{theorem}

\begin{remark}\label{rem1}
In a preprint \cite{fu2}, the author extended the invariant $T$ to
nanophrases over any $\alpha$. However in this paper we only use this
invariant for nanophrases over $\alpha_{0}$ and nanophrases over the
one element set. 
\end{remark}

\subsection{The invariant $S_{o}$.}
In \cite{gi} A.Gibson defined a homotopy invariant of nanophrases over the 
one-element set which is stronger than the invariant $T$ for
nanophrases over the one-element set. In this subsection we introduce
Gibson's $S_{o}$ invariant.\par
First we define some notations. 
Let $(\mathcal{A},P = (w_{1}|\cdots|w_{k}))$ be a nanophrase over
the one-element set.
For a letter $A \in \mathcal{A}_{i} := 
\{A \in \mathcal{A} | Card(w_{i}^{-1}(A))=2  \}$, 
we define
$l_{j}(A) \in \mathbb{Z} / 2\mathbb{Z}$ as follows : When we write $P$ as
$xAyAz$ where $x$, $y$ and $z$ are words in $\mathcal{A}$ 
possibly including "$|$" character, $l_{j}(A)$ is modulo 2 of 
the number of letters which appear exactly once in $y$ and once in
the j-th component of the phrase $P$. Then we define 
$l(A)\in (\mathbb{Z} / 2\mathbb{Z})^{k}$ by 
$$l(A):=(l_{1}(A),l_{2}(A),\cdots,l_{k}(A)).$$
Let $v$ be a vector in $(\mathbb{Z} / 2\mathbb{Z})^{k}$. Then
we define $d_{j}(v) \in \mathbb{Z}$ 
by $$d_{j}(v):=Card(\{ A \in \mathcal{A}_{j}|l(A)=v \}),$$
and we define $B_{j}(P) \in 2^{(\mathbb{Z}/2\mathbb{Z})^{k}}$ by
$$B_{j}(P):=\{v \in (\mathbb{Z}/2\mathbb{Z})^{k} \setminus \{0\}|
d_{j}(v)=1 \ mod \ 2 \}.$$
Then we define the $S_{o}(P) \in (2^{(\mathbb{Z}/2\mathbb{Z})^{k}})^{k}$ by
$$S_{o}(P) := (B_{1}(P),B_{2}(P), \cdots ,B_{k}(P)).$$
\begin{theorem}[Gibson \cite{gi}]
$S_{o}$ is a homotopy invariant of nanophrases over the one-element set.
\end{theorem}

\subsection{The invariant $U_{L}$.}
In this section we introduce a new invariant of nanophrases.\par
First we prepare some notations. Since the set $\alpha$ is a finite set,
we obtain following orbit decomposition of the $\tau$ : 
$\alpha/\tau = \{ \widehat{a_{i_1}}, \widehat{a_{i_2}}, \cdots 
,\widehat{a_{i_l}}, \widehat{a_{i_{l+1}}},\cdots, 
\widehat{a_{i_{l+m}}} \}$, where 
$\widehat{a_{i_j}}:= \{ a_{i_j}, \tau(a_{i_j}) \}$ such that 
$Card(\widehat{a_{i_j}})=2$ for all $j \in \{1,\cdots,l \}$ and 
$Card(\widehat{a_{i_j}})=1$ for all $j \in \{l+1,\cdots,l+m \}$ (we fix
a complete representative system $crs(\alpha/\tau):=
\{ a_{i_1}, a_{i_2}, \cdots ,a_{i_l}, a_{i_{l+1}},\cdots, a_{i_{l+m}} \}$ 
which satisfy the above condition). Let $L$ be a subset of 
$crs(\alpha/\tau)$. For a nanophrase $(\mathcal{A},P)$ over $\alpha$,
we define a nanophrase $U_{L}((\mathcal{A},P))$ over $L \cup \tau(L)$
as follows:
deleting all letters $A \in \mathcal{A}$ such that $|A| \not\in L \cup \tau(L)$
from both $\mathcal{A}$ and $P$.
\begin{proposition}
$U_{L}$ is a homotopy invariant of nanophrases.
\end{proposition}     
\begin{proof}
First, isomorphism does not change $U_{L}(P)$ up to isomorphic is clear.\par
Consider the first homotopy move 
$$P_{1}:=(\mathcal{A}, (xAAy)) \longrightarrow P_{2}:=(\mathcal{A} \setminus
\{A\}, (xy))$$
where $x$ and $y$ are words on $\mathcal{A}$, 
possibly including "$|$" character.
Suppose $|A| \in L \cup \tau(L)$. Then
$$U_{L}(P_{1})=x_{L}AAy_{L} \simeq x_{L}y_{L} = U_{L}(P_{2})$$
where $x_{L}$ and $y_{L}$ are words which obtained by deleting all letters
$X \in \mathcal{A}$ such that $X \not\in L \cup \tau(L)$ from $x$ and
$y$ respectively.\par
Suppose $|A| \not\in L \cup \tau(L)$. Then
$$U_{L}(P_{1})=x_{L}y_{L}=U_{L}(P_{2}).$$
So the first homotopy move does not change the homotopy class of $U_{L}(P)$.\par
Consider the second homotopy move
$$P_{1}:=(\mathcal{A}, (xAByBAz)) \longrightarrow 
P_{2} := (\mathcal{A} \setminus \{A,B\}, (xyz))$$
where $|A|=\tau(|B|)$, and $x$, $y$ and $z$ are words on $\mathcal{A}$
possibly including "$|$" character.
Suppose $|A| \in L \cup \tau(L)$. Then $|B| \in L \cup \tau(L)$ 
since $|A|= \tau(|B|)$. So
$$U_{L}(P_{1})=x_{L}ABy_{L}BAz_{L} \simeq x_{L}y_{L}z_{L} = U_{L}(P_{2}).$$
Suppose $|A| \not\in L \cup \tau(L)$. Then $|B| \not\in L \cup \tau(L)$ 
since $|A|= \tau(|B|)$. So
$$U_{L}(P_{1})= x_{L}y_{L}z_{L} = U_{L}(P_{2}).$$
By the above, the second homotopy move does not change the homotopy 
class of $U_{L}(P)$.\par
Consider the third homotopy move 
$$P_{1}:= (\mathcal{A},(xAByACzBCt)) \rightarrow 
P_{2}:=(\mathcal{A},(xBAyCAzCBt))$$
where $|A|=|B|=|C|$, and $x$, $y$, $z$ and $t$ are words on 
$\mathcal{A}$ possibly including "$|$" character.
Suppose $|A| \in L \cup \tau(L)$. Then $|B|, |C| \in L \cup \tau(L)$ since
$|A|=|B|=|C|$. So we obtain
$$U_{L}(P_{1}) = x_{L}ABy_{L}ACz_{L}ACt_{L} \simeq x_{L}BAy_{L}CAz_{L}CBt_{L}
=U_{L}(P_{2}).$$
Suppose $|A| \not\in L \cup \tau(L)$. 
Then $|B|, |C| \not\in L \cup \tau(L)$ since
$|A|=|B|=|C|$. So we obtain
$$U_{L}(P_{1}) = x_{L}y_{L}z_{L}t_{L} = U_{L}(P_{2}).$$
So the third homotopy move does not change 
the homotopy class of $U_{L}(P)$.\par
By the above, $U_{L}$ is a homotopy invariant of nanophrases.   
\end{proof}     

\section{Homotopy Classification of  \'Etale Phrases.}
In this section we classify \'etale phrases with less than or 
equal to three letters up to homotopy. First we recall lemmas 
in \cite{fu1}.

\begin{lemma}\label{plem1}
Let $\beta$ be $\tau$-invariant subset of $\alpha$. If two nanophrases
 over $\beta$ are homotopic in the class of nanophrases over $\alpha$,
 then they are homotopic in the class of nanophrases over $\beta$. 
\end{lemma}

\begin{lemma}\label{plem2}
Let $P_1=(w_1|w_2|\cdots|w_k)$ and  $P_2=(v_1|v_2|\cdots|v_k)$ be 
nanophrases of length $k$ over $\alpha$. If $P_1$ and $P_2$ are
homotopic as nanophrases, then $w_i$ and $v_i$ are homotopic as 
\'etale words for all $i \in \{1,2,,\cdots,k\}$.
\end{lemma}

Next we prepare some notations. Let $\alpha$ be an alphabet 
endowed with an involution $\tau:\alpha \rightarrow \alpha$.
Then we set \\
$P^{1,1;l_{1},l_{2}}_{a}
:=(\emptyset|\cdots|\emptyset|\stackrel{l_{1}}{\check{a}}|\emptyset|\cdots|\emptyset|\stackrel{l_{2}}{\check{a}}|\emptyset|\cdots|\emptyset)$,\\
$P^{3;l}_{a}
:=(\emptyset|\cdots|\emptyset|\stackrel{l}{\check{a^{3}}}|\emptyset|\cdots|\emptyset)$,\\
$P^{2,1;l_{1},l_{2}}_{a}
:=(\emptyset|\cdots|\emptyset|\stackrel{l_{1}}{\check{a^{2}}}|\emptyset|\cdots|
\emptyset|\stackrel{l_{2}}{\check{a}}|\emptyset|\cdots|\emptyset)$,\\
$P^{1,2;l_{1},l_{2}}_{a}
:=(\emptyset|\cdots|\emptyset|\stackrel{l_{1}}{\check{a}}|\emptyset|\cdots|
\emptyset|\stackrel{l_{2}}{\check{a^{2}}}|\emptyset|\cdots|\emptyset)$,\\
$P^{1,1,1;l_{1},l_{2},l_{3}}_{a}
:=(\emptyset|\cdots|\emptyset|\stackrel{l_{1}}{\check{a}}|\emptyset|\cdots|\stackrel{l_{2}}{\check{a}}|\emptyset|
\cdots|\emptyset|\stackrel{l_{3}}{\check{a}}|\emptyset|\cdots|\emptyset)$,\\
where $a \in \alpha$ and $l$, $l_{1}$,$l_{2}$,$l_{3} \in \hat{k}$ with 
$l_{1} < l_{2} < l_{3}$.  
Note that if $a=\tau(a)$, then $P^{3;l}_{a}$ is 
homotopic to the nanophrase $(\emptyset)_{k}:=(\emptyset|\cdots|\emptyset)$.
So when we use the notation $P^{3;l}_{a}$, 
we always assume that $a \neq \tau(a)$. 

\begin{remark}
For two different integers $k_{1}$ and $k_{2}$, 
an \'etale phrase of length $k_{1}$ and 
an \'etale phrase of length $k_{2}$
are not homotopic each other.
So we do not write length of phrases in above notations.  
\end{remark}

Now we describe the main results of this paper.

\begin{theorem}\label{mthm}
Let $P$ be a multiplicity-one-free \'etale phrase over $\alpha$
with less than or equal to three letters. Then $P$ is either homotopic
to $(\emptyset)_{k}$ or homotopic to one of the following 
\'etale phrases: $P^{1,1;l_{1},l_{2}}_{a}$, $P^{3;l}_{a}$, 
$P^{2,1;l_{1},l_{2}}_{a}$, $P^{1,2;l_{1},l_{2}}_{a}$, 
$P^{1,1,1;l_{1},l_{2},l_{3}}_{a}$ for some $l_{1}$, $l_{2}$, $l_{3} \in \hat{k}$ 
and $a \in \alpha$. Moreover 
$P^{1,1;l_{1},l_{2}}_{a}$, $P^{3;l}_{a}$, 
$P^{2,1;l_{1},l_{2}}_{a}$, $P^{1,2;l_{1},l_{2}}_{a}$, 
$P^{1,1,1;l_{1},l_{2},l_{3}}_{a}$
are homotopic if and only if they are equal.    
\end{theorem}
 
To prove this theorem, we prepare following lemmas.

\begin{lemma}\label{lem1}
\'Etale phrases $P^{1,1;l_{1},l_{2}}_{a}$, $P^{3;l}_{a}$, 
$P^{2,1;l_{1},l_{2}}_{a}$, $P^{1,2;l_{1},l_{2}}_{a}$ and 
$P^{1,1,1;l_{1},l_{2},l_{3}}_{a}$ are not homotopic to $(\emptyset)_{k}$.
\end{lemma}

\begin{lemma}\label{lem2}
If $a$ is not equal to $b$, then $P^{X_{1};Y_{1}}_{a}$ and $P^{X_{2};Y_{2}}_{b}$
are not homotopic for all $(X_{1};Y_{1}), (X_{2};Y_{2}) \in \{(1,1;l_{1},l_{2})
,(3;l),(2,1;l_{1},l_{2}),(1,2;l_{1},l_{2}),(1,1,1;l_{1},l_{2},l_{3})\}$.
\end{lemma}

\begin{lemma}\label{lem3}
Two \'etale phrases  $P^{X_{1};Y_{1}}_{a}$ and $P^{X_{2};Y_{2}}_{a}$ are
homotopic if and only if  $(X_{1};Y_{1})$ is equal to $(X_{2};Y_{2})$.
\end{lemma}

If we show above lemmas, then we obtain the main theorem.
We prove these lemmas in the next section.

\section{Proof of Lemmas.}
In this section, we prove Lemma \ref{lem1}, Lemma \ref{lem2} 
and Lemma\ref{lem3}.
\subsection{Proof of the Lemma \ref{lem1}.}
The first claim of this lemma is easily checked.
We show the second part of the lemma.\\
$\bullet$ The case $P^{1,1;l_{1},l_{2}}_{a} \not\simeq (\emptyset)_{k}$.\par
In this case, component length vector 
$$w(P^{1,1;l_{1},l_{2}}_{a}) = \mathbf{e}_{l_{1}} + \mathbf{e}_{l_{2}}$$
where $\mathbf{e}_{j} = (0,\cdots,0,\stackrel{j}{\check{1}},0,\cdots,0)$.
On the other hand
$$w((\emptyset)_{k})=0.$$
So we obtain $P^{1,1;l_{1},l_{2}}_{a} \not\simeq (\emptyset)_{k}$.\\
$\bullet$ The case $P^{3;l}_{a} \not\simeq (\emptyset)_{k}$.\par
By the Theorem \ref{tu}, $a^{3}=aaa$ with $a \neq \tau(a)$ is
not homotopic to 
empty nanoword $\emptyset$. Combining this fact and Lemma \ref{plem2},
we obtain  $P^{3;l}_{a} \not\simeq (\emptyset)_{k}$.\\
$\bullet$ The case $P^{2,1;l_{1},l_{2}}_{a} \not\simeq (\emptyset)_{k}$.\par
By Lemma \ref{plem1}, we can assume $\alpha = \{a,\tau(a)\}$.
In this case  
$$T(P^{2,1;l_{1},l_{2}}_{a})=\mathbf{e}_{l_{1}} \neq 0$$
(both the case $a=\tau(a)$ and the case $a \neq \tau(a)$).
So we obtain $P^{2,1;l_{1},l_{2}}_{a} \not\simeq (\emptyset)_{k}$.\\ 
$\bullet$ The case  $P^{1,2;l_{1},l_{2}}_{a} \not\simeq (\emptyset)_{k}$.\par
In this case  
$$T(P^{1,2;l_{1},l_{2}}_{a})=\mathbf{e}_{l_{2}} \neq 0$$
(both the case $a=\tau(a)$ and the case $a \neq \tau(a)$).
So we obtain $P^{1,2;l_{1},l_{2}}_{a} \not\simeq (\emptyset)_{k}$.\\
$\bullet$ The case 
$P^{1,1,1;l_{1},l_{2},l_{3}}_{a} \not\simeq (\emptyset)_{k}$.\par 
In this case 
$$l_{P^{1,1,1;l_{1},l_{2},l_{3}}_{a}}(l_{1},l_{2})=a \in \pi.$$
Note that $a$ is not equal to the unit element $1$ in $\pi$.
So we obtain $P^{1,1,1;l_{1},l_{2},l_{3}}_{a} \not\simeq (\emptyset)_{k}$.\par
Now we finished prove the lemma. 

\subsection{Proof of the Lemma \ref{lem2}.}
By Lemma \ref{plem1}, we can assume $\alpha=\{a,\tau(a),b,\tau(b)\}$.\par
Suppose $\hat{a} \neq \hat{b}$. Let $crs(\alpha / \tau) = \{a,b\}$
and $L = \{a\}$.
Then 
$$U_{L}(P^{X_{1};Y_{1}}_{a}) = P^{X_{1};Y_{1}}_{a}.$$
On the other hand, 
$$U_{L}(P^{X_{2};Y_{2}}_{b}) = (\emptyset)_{k}.$$
So by Lemma \ref{lem1}, we obtain $P^{X_{1};Y_{1}}_{a}$ is not homotopic
to $P^{X_{2};Y_{2}}_{b}$ for all 
$(X_{1};Y_{1}), (X_{2};Y_{2})$ $\in \{(1,1;l_{1},l_{2})
,(3;l),(2,1;l_{1},l_{2}),(1,2;l_{1},l_{2}),(1,1,1;l_{1},l_{2},l_{3})\}$.\par
Suppose $\hat{a} = \hat{b}$. By the assumption $a$ is not equal to $b$,
$Card(\hat{a})=2$ and $b=\tau(a)$.\\
$\bullet$ On $P^{1,1;l_{1},l_{2}}_{a} \not\simeq P^{1,1;m_{1},m_{2}}_{\tau(a)}$.\par
In this case 
$$l_{P^{1,1;l_{1},l_{2}}_{a}}(l_{1},l_{2}) = a.$$
On the other hand,
$$l_{P^{1,1;m_{1},m_{2}}_{\tau(a)}}(l_{1},l_{2}) =
\begin{cases} \tau(a) = a^{-1} \ ( \ if \ (l_{1},l_{2})=(m_{1}.m_{2}) \ ), \\ 
              1 \ \ (otherwise).
\end{cases}$$  
So $l_{P^{1,1;l_{1},l_{2}}_{a}}(l_{1},l_{2})$ is not equal to 
$l_{P^{1,1;m_{1},m_{2}}_{\tau(a)}}(l_{1},l_{2})$. 
So we obtain 
$P^{1,1;l_{1},l_{2}}_{a} \not\simeq P^{1,1;m_{1},m_{2}}_{\tau(a)}$.\\ 
$\bullet$ On $P^{1,1;l_{1},l_{2}}_{a} \not\simeq P^{X;Y}_{\tau(a)}$ for all 
$(X;Y) \neq (1,1;m_{1},m_{2})$.\par
In this case,
$$w(P^{1,1;l_{1},l_{2}}_{a}) = \mathbf{e}_{l_{1}} + \mathbf{e}_{l_{2}},$$ 
and
$$w( P^{X;Y}_{\tau(a)}) = 0.$$
So we obtain  $P^{1,1;l_{1},l_{2}}_{a} \not\simeq P^{X;Y}_{\tau(a)}$.\\
$\bullet$ On $P^{3;l}_{a} \not\simeq P^{X;Y}_{\tau(a)}$ 
 for all $(X;Y) \neq (1,1;m_{1},m_{2})$.\par
This is obtained from Theorem \ref{tu} and Lemma \ref{plem2}.\\
$\bullet$ On $P^{2,1;l_{1},l_{2}}_{a} \not\simeq P^{X;Y}_{\tau(a)}$
for all 
$(X;Y) \neq (1,1;m_{1},m_{2}), (3;m)$.\par
In this case,
$$l_{P^{2,1;l_{1},l_{2}}_{a}}(l_{1},l_{2}) = a^{2}.$$
On the other hand
$$l_{P^{2,1;m_{1},m_{2}}_{\tau(a)}}(l_{1},l_{2}) = 
\begin{cases} 
\tau(a)^{2} = a^{-2} \ ( \ if \ (l_{1},l_{2})=(m_{1}.m_{2}) \ ), \\ 
              1 \ \ (otherwise),
\end{cases}$$
$$l_{P^{1,2;m_{1},m_{2}}_{\tau(a)}}(l_{1},l_{2}) = 
\begin{cases} 
\tau(a)^{2} = a^{-2} \ ( \ if \ (l_{1},l_{2})=(m_{1}.m_{2}) \ ), \\ 
              1 \ \ (otherwise),
\end{cases}$$
and
$$l_{P^{1,1,1;m_{1},m_{2},m_{3}}_{\tau(a)}}(l_{1},l_{2}) = 
\begin{cases} 
\tau(a) = a^{-1} \ ( \ if \ \exists(m_{i},m_{j})=(l_{1}.l_{2}) \ ), \\ 
              1 \ \ (otherwise).
\end{cases}$$
So we obtain  $P^{2,1;l_{1},l_{2}}_{a} \not\simeq P^{X;Y}_{\tau(a)}$.\\
$\bullet$ On $P^{1,2;l_{1},l_{2}}_{a} \not\simeq P^{X;Y}_{\tau(a)}$
for all 
$(X;Y) \neq (1,1;m_{1},m_{2}), (3;m),(2,1;m_{1},m_{2})$.\par 
In this case,
$$l_{P^{1,2;l_{1},l_{2}}_{a}}(l_{1},l_{2}) = a^{2}.$$
On the other hand
$$l_{P^{1,2;m_{1},m_{2}}_{\tau(a)}}(l_{1},l_{2}) = 
\begin{cases} 
\tau(a)^{2} = a^{-2} \ ( \ if \ (l_{1},l_{2})=(m_{1}.m_{2}) \ ), \\ 
              1 \ \ (otherwise),
\end{cases}$$
and
$$l_{P^{1,1,1;m_{1},m_{2},m_{3}}_{\tau(a)}}(l_{1},l_{2}) = 
\begin{cases} 
\tau(a) = a^{-1} \ ( \ if \ \exists(m_{i},m_{j})=(l_{1}.l_{2}) \ ), \\ 
              1 \ \ (otherwise).
\end{cases}$$
So we obtain  $P^{1,2;l_{1},l_{2}}_{a} \not\simeq P^{X;Y}_{\tau(a)}$.\\
$\bullet$ On $P^{1,1,1;l_{1},l_{2},l_{3}}_{a} \not\simeq
P^{1,1,1;m_{1},m_{2},m_{3}}_{\tau(a)}$.\par
In this case
$$l_{P^{1,1,1;l_{1},l_{2},l_{3}}_{a}}(l_{1},l_{2})=a,$$
and
$$l_{P^{1,1,1;m_{1},m_{2},m_{3}}_{\tau(a)}}(l_{1},l_{2}) = 
\begin{cases} 
\tau(a) = a^{-1} \ ( \ if \ \exists(m_{i},m_{j})=(l_{1}.l_{2}) \ ), \\ 
              1 \ \ (otherwise).
\end{cases}$$
So we obtain  $P^{1,1,1;l_{1},l_{2},l_{3}}_{a} \not\simeq
P^{1,1,1;m_{1},m_{2},m_{3}}_{\tau(a)}$.\par
Now we have completed the proof of Lemma \ref{lem2}.

\subsection{Proof of the Lemma \ref{lem3}.}
$\bullet$ On $P^{1,1;l_{1},l_{2}}_{a}$.\par 
In this case
$$w(P^{1,1;l_{1},l_{2}}_{a})= \mathbf{e}_{l_{1}}+\mathbf{e}_{l_{2}},$$
and
$$w(P^{X;Y}_{a})=0$$
for all $(X;Y) \neq (1,1;m_{1},m_{2})$.
So we obtain  $P^{1,1;l_{1},l_{2}}_{a} \simeq P^{X;Y}_{a}$ 
if and only if  $(X;Y)=(1,1;l_{1},l_{2})$.\\
$\bullet$ On $P^{3;l}_{a}$.\par
By Theorem \ref{tu} and Lemma \ref{plem2},
We obtain  $P^{3;l}_{a} \simeq P^{X;Y}_{a}$ 
if and only if  $(X;Y)=(3;l)$.\\
$\bullet$ On $P^{2,1;l_{1},l_{2}}_{a}$.\par
The case $P^{2,1;l_{1},l_{2}}_{a} \not\simeq P^{2,1;m_{1},m_{2}}_{a}$ if 
$(l_{1},l_{2}) \neq (m_{1},m_{2})$.\par
If $a \neq \tau(a)$, then
$$l_{P^{2,1;l_{1},l_{2}}_{a}}(l_{1},l_{2})=a^{2},$$
and 
$$l_{P^{2,1;m_{1},m_{2}}_{a}}(l_{1},l_{2})= 1 \neq a^{2}.$$
So we obtain  $P^{2,1;l_{1},l_{2}}_{a} \not\simeq P^{2,1;m_{1},m_{2}}_{a}$ if 
$(l_{1},l_{2}) \neq (m_{1},m_{2})$.\par
If $a = \tau(a)$, then by Lemma \ref{plem1} 
we can assume $\alpha = \{a\}$.
So we can use Gibson's $S_{o}$ invariant.
In this case
$$S_{o}(P^{2,1;l_{1},l_{2}}_{a})=
(\emptyset,\cdots,\emptyset,\stackrel{l_{1}}{\check{\{\mathbf{e}_{l_{2}}\}}},
\emptyset, \cdots ,\emptyset),$$
and
$$S_{o}(P^{2,1;m_{1},m_{2}}_{a})=
(\emptyset,\cdots,\emptyset,\stackrel{m_{1}}{\check{\{\mathbf{e}_{m_{2}}\}}},
\emptyset, \cdots ,\emptyset).$$
So we obtain 
$P^{2,1;l_{1},l_{2}}_{a} \not\simeq P^{2,1;m_{1},m_{2}}_{a}$ if 
$(l_{1},l_{2}) \neq (m_{1},m_{2})$.\par
The case $P^{2,1;l_{1},l_{2}}_{a} \not\simeq P^{1,2;m_{1},m_{2}}_{a}$.\par
If $a \neq \tau(a)$ and $(l_{1},l_{2}) \neq (m_{1},m_{2})$, then
$$l_{P^{2,1;l_{1},l_{2}}_{a}}(l_{1},l_{2})=a^{2},$$ 
and
$$l_{P^{1,2;m_{1},m_{2}}_{a}}(l_{1},l_{2})=1.$$
If $a \neq \tau(a)$ and $(l_{1},l_{2}) = (m_{1},m_{2})$, then
$$T(P^{2,1;l_{1},l_{2}}_{a})=\mathbf{e}_{l_{1}},$$
and
$$T(P^{1,2;l_{1},l_{2}}_{a})= - \mathbf{e}_{l_{2}}.$$
Since $l_{1}$ is not equal to $l_{2}$,
$$T(P^{2,1;l_{1},l_{2}}_{a}) \neq  T(P^{1,2;l_{1},l_{2}}_{a}).$$
If $a = \tau(a)$, then 
$$S_{o}(P^{2,1;l_{1},l_{2}}_{a})=
(\emptyset,\cdots,\emptyset,\stackrel{l_{1}}{\check{\{\mathbf{e}_{l_{2}}\}}}
,\emptyset, \cdots ,\emptyset),$$
and
$$S_{o}(P^{1,2;m_{1},m_{2}}_{a})=
(\emptyset,\cdots,\emptyset,\stackrel{m_{2}}{\check{\{\mathbf{e}_{m_{1}}\}}}
,\emptyset, \cdots ,\emptyset).$$
So if $P^{2,1;l_{1},l_{2}}_{a} \simeq P^{1,2;m_{1},m_{2}}_{a}$, 
then $(l_{1},l_{2})=(m_{2},m_{1})$.
However this contradict the 
assumption $l_{1} < l_{2}$ and $m_{1} < m_{2}$.
By the above, we obtain 
$P^{2,1;l_{1},l_{2}}_{a} \not\simeq P^{1,2;m_{1},m_{2}}_{a}$. \par
The case $P^{2,1;l_{1},l_{2}}_{a} \not\simeq P^{1,1,1;m_{1},m_{2},m_{3}}_{a}$.\par
In this case,
$$lk(P^{1,1,1;m_{1},m_{2},m_{3}}_{a})= 
\begin{cases} 
(1,\cdots,1,a^{2},1,\cdots,1) \ ( \ if \ a \neq \tau(a) \ ), \\ 
              (1,\cdots,1) \ \ ( \ if \ a = \tau(a) \ ).
\end{cases}$$
On the other hand,
$$lk( P^{1,1,1;m_{1},m_{2},m_{3}}_{a})=
(1,\cdots,1,a,1,\cdots,1,a,1,\cdots,1,a,1,\cdots,1).$$
So we obtain $lk(P^{2,1;l_{1},l_{2}}_{a}) \neq lk(P^{1,1,1;m_{1},m_{2},m_{3}}_{a})$.
This implies 
$P^{2,1;l_{1},l_{2}}_{a} \not\simeq P^{1,1,1;m_{1},m_{2},m_{3}}_{a}$.\\
$\bullet$ On $P^{1,2;l_{1},l_{2}}_{a}$.\par
This case proved similarly as the case on $P^{2,1;l_{1},l_{2}}_{a}$.\\
$\bullet$ On $P^{1,1,1;l_{1},l_{2},l_{3}}_{a}$.\par
In this case,
$$l_{P^{1,1,1;l_{1},l_{2},l_{3}}_{a}}(i,j) = 
\begin{cases} 
 a \ ( \ if \ (i,j)=(l_{1},l_{2}),(l_{1},l_{3}),(l_{2},l_{3}) \ ), \\ 
              1 \ \ (otherwise).
\end{cases}$$
So we obtain  $P^{1,1,1;l_{1},l_{2},l_{3}}_{a} \simeq 
P^{1,1,1;m_{1},m_{2},m_{3}}_{a}$ if and only if 
$(l_{1},l_{2},l_{3})=(m_{1},m_{2},m_{3})$.\par
Now we have completed the proof of Lemma \ref{lem3}.

\vspace{0.3cm}
Department of Mathematics, Hokkaido University

Sapporo 060-0810, Japan

e-mail: fukunaga@math.sci.hokudai.ac.jp

\end{document}